\documentclass[nospthms,reqno]{svjour2}

\usepackage{amssymb, tikz}
\usepackage{mathrsfs}
\usepackage{dsfont}
\usepackage{amsfonts,enumitem, enumerate, xspace, amsthm}
\usepackage{changebar}
\usepackage{hyperref}
\usepackage{pdfsync}
\usepackage{color}
\usepackage{bm}
\usepackage{amsmath}

\numberwithin{equation}{section}

\newtheorem{thm}{Theorem}[section]
\newtheorem{lem}[thm]{Lemma}
\newtheorem{cor}[thm]{Corollary}

\newtheorem{prop}[thm]{Proposition}

\newtheorem{defin}{Definition}[section]

\newcommand\eps{\epsilon}

\def\om{{\omega}}
\def\Om{{\Omega}}
\def\dsum{\displaystyle\sum}
\def\dlim{\displaystyle\lim}
\def\dint{\displaystyle\int}

\allowdisplaybreaks

\newcommand\Exp{\operatorname{Exp}}

\begin{document}

\title{Fick law and sticky Brownian motions}

\author{Thu Dang Thien Nguyen}

\institute{Thu Dang Thien Nguyen \at
             Gran Sasso Science Institute, Viale Francesco Crispi, 7, L'Aquila  67100, Italy\\
			 Department of Mathematics, University of Quynhon, Quy Nhon, Vietnam\\
             \email{thu.nguyen@gssi.it}           
}

\date{Received: \today / Accepted: date}

\maketitle

\begin{abstract}
{We consider an interacting particle system in the interval $[1,N]$ with reservoirs at the boundaries. While the dynamics in the channel is the simple symmetric exclusion process, the reservoirs are also particle systems which interact with the given system by exchanging particles. In this paper we study the case where the size of each reservoir is the same as the size of the channel . We will prove that the hydrodynamic limit equation is the heat equation with boundary conditions which relate first and second spatial derivatives at the boundaries for which we will prove the existence and uniqueness of weak solutions. The propagation of chaos property can also be derived.
}
\end{abstract}
\maketitle

\section{Introduction}
The Fick law states that the mass flux is proportional to minus the gradient of the mass density, it is the analogue of the Fourier law for heat conduction.
Many efforts have been devoted to its rigorous derivation  from ``microscopic'' particle systems.  The first paper in this direction is \cite{GKMP} where the Fick law has been derived in the hydrodynamic limit from  the simple symmetric exclusion process (SSEP) complemented by birth-death processes at the extremes.
Such processes simulate density reservoirs which  force the density at the boundaries to given values, say $\rho_{\pm}$.  Extensions to more general systems have been obtained afterwards where however the action of the reservoirs has always been represented by birth-death processes. The case where the boundary densities $\rho_{\pm}$ might be time-dependent has been studied more recently in the context of the SSEP. In such papers time dependence
arises when either the  birth-death processes are suitably ``weak'', \cite{DPTV1},
\cite{DPTV2},\cite{DPTV3},\cite{DPTV4}, or when
the rates are directly time-dependent as in \cite{DO} where
the so called adiabatic limit has been derived.

To model the action of reservoirs as birth-death processes is clearly an approximation. Physically reservoirs are also particle systems which interact with our given system $\mathcal S$ with which they exchange particles: the idea behind is that (i) reservoirs are ``very large'' so that their density does not have significant variations, (ii) the reservoirs have a ``very fast'' approach to equilibrium so that their interaction with $\mathcal S$ looks always the same.
In this and in a successive paper we plan to investigate these issues when the system $\mathcal S$ evolves as in the SSEP and the interaction with the reservoirs is of mean field type, this assumption is to implement (ii).  We will see that by varying the size of the reservoirs we obtain different regimes. In this paper we study the case where the size of each reservoir is the same as the size of $\mathcal S$.  We will prove that the limit hydrodynamic equation is the heat equation with boundary conditions which relate first and second spatial derivatives at the boundaries for which we will prove existence and uniqueness of weak solutions.  The result for weak solutions has been considered by Stroock and Williams, \cite{SW}, and by Peskir, \cite{P}, for the system in the half line. The solution of the limit equation has a probabilistic representation in terms of sticky Brownian motions and our result proves convergence of sticky random walks to sticky  Brownian motions extending the proof of \cite{A} which referred to motion on the half line.


\vskip2cm

\section{Model and main results}

\subsection{The model}

We will study a jump Markov process with state space $\Om=\{0,1\}^{\Lambda_N} \times \{0,..,M\} \times \{0,..,M\}$, where $\Lambda_N=[1,N]\cap \mathbb{N}$, $M$ and $N$ are positive integers, we will eventually set $M=N$.  Thus an element $\om$ of $\Om$ is a sequence $\eta(x)$, $x=1,..,N$, complemented by two integers $n_-$ and $n_+$ both in $[0,M]$: $\eta(x)\in \{0,1\}$,
$\eta(x)=1$ means that there is a particle at $x$ and $\eta(x)=0$ that $x$ is empty. If $n_+=k$ we will then say that there are $k$ particles in $\mathcal S_+$, likewise if  $n_-=k$ then there are $k$ particles in $\mathcal S_-$, such a terminology will become clear in a while.

The generator $\mathbf{L}$ of the process is defined as
 \begin{equation}
 \label{2.1}
 \mathbf{L}f(\om) =\frac 12 \sum_{x=1}^{N-1}[ f(\om^{x,x+1})-f(\om)]
 + c_N(\om)[ f(\om^N)-f(\om)] +  c_{1}(\om)[ f(\om^{1})-f(\om)]
 \end{equation}
where $\om^{x,x+1}$ is obtained from $\om$ by exchanging $\eta(x)$ and $\eta(x+1)$; 
 \begin{equation}
 \label{2.2}
 c_N(\om) = \frac 12\Big( 1-\frac{n_+}{M}\Big),\quad \text{if $\eta(N)=1$}
 \end{equation}
and $\om^N$ has $\eta^N(N)=0$ and $n_+^N= n_++1$: we then say that a particle has jumped from $N$ to $\mathcal S_+$; 
 \begin{equation}
 \label{2.3}
 c_N(\om) =  \frac 12\,\frac{n_+}{M},\quad \text{if  $\eta(N)=0$}
 \end{equation}
and $\om^N$ has $\eta^N(N)=1$ and $n_+^N= n_+-1$: we then say that a particle has jumped from $\mathcal S_+$ to $N$. $c_1(\om)$ is defined analogously. Notice that if $n_+=M$ then $c_N(\om)=0$ in \eqref{2.2} so that we always have $n_+ \le M$, likewise $n_+\ge 0$ (same for $n_-$).

The SSEP with reservoirs mentioned in the introduction is the above process where $n_{\pm}$ do not change, i.e.\ we redefine $\om^N$ and $\om^1$ by setting $n_+^N= n_+$ and $n_-^N=n_-$.  This means that   particles jump from $N$ to $\mathcal S_+$ at constant rate  \eqref{2.2} and a new particle enters at $N$ with rate  \eqref{2.3} (analogous processes occurring at $1$).

The above process can be realized in a larger space:  $\Om^*=\{0,1\}^{\Lambda_N} \times \{0,1\}^{\mathcal S_-} \times \{0,1\}^{\mathcal S_+}$: the elements $\om$ of $\Om^*$ are sequences $\eta(x)$, $x \in \Lambda_N \cup \mathcal S_-\cup \mathcal S_+$, and the generator is
 \begin{align}
 \mathbf{L}^*f(\om) =\frac 12 \sum_{x=1}^{N-1}[ f(\om^{x,x+1})-f(\om)]
 + &\frac 1{2M} \sum_{z \in \mathcal S_-}[ f(\om^{1,z})-f(\om)]\notag\\
 + & \frac 1{2M} \sum_{z \in \mathcal S_+}[ f(\om^{N,z})-f(\om)] \label{2.4}
 \end{align}
This is the usual stirring process but with rates slowed down for exchanges between $N$ and sites in $\mathcal S_+$ and between 1 and sites in $\mathcal S_-$.

One can readily see that the marginal over the variables $\eta(x), x \in \Lambda_N$, and
   $$
   n_{\pm} :=\sum_{z\in \mathcal S_{\pm}} \eta(z)
   $$
has the law of the process with generator $\mathbf{L}$.  This remains true also if we allow autonomus evolutions (preserving the number of particles) in  $\mathcal S_-$ and $\mathcal S_+$ due to the mean field nature of the interaction between $N$ and $\mathcal S_+$ and $1$ and $\mathcal S_-$.  As discussed later for large $M$ and ``hydrodynamic times'' (i.e. which scale as $N^2$) the systems in $\mathcal S_-$ and $\mathcal S_+$ act as reservoirs for the system in $\Lambda_N$; however for times longer
than the hydrodynamical ones the densities in  $\mathcal S_-$ and $\mathcal S_+$ change till equilibration.

\vskip.5cm

\subsection{The hydrodynamic equations}

In the sequel we  identify $\eps=1/N$. Denote $\bar{\Lambda}_N:=[0,N+1]\cap\mathbb{N}$. We write for $x\in\bar{\Lambda}_N$,
$$ \rho^{\eps}(x,t) = \mathbf{E}^\eps\Big[ \eta(x,t)\Big],$$
where $\eta(0,t):=\frac{n_-(t)}{M}, \,\eta(N+1,t):=\frac{n_+(t)}{M}$ and $\mathbf{E}^\eps$ stands for the expectation with respect to the process with generator \eqref{2.1}, the initial distribution will be specified later. Sometimes we use $\rho^{\eps}_0(x)$ instead of $\rho^{\eps}(x,0)$.

We denote $\rho^{\eps}_-(t):=\rho^{\eps}(0,t),\quad  \rho^{\eps}_+(t):=\rho^{\eps}(N+1,t)$.

Then for $x\in \Lambda_N$, we have
\begin{equation} \label{2.6}
 \frac{d}{dt}\rho^{\eps}(x,t) =  \frac 12 \Big( \rho^{\eps}(x-1,t)+\rho^{\eps}(x+1,t)
 -2\rho^{\eps}(x,t)\Big).
 \end{equation}

 We also have  
\begin{equation}
 \label{2.7}
 \frac{d}{dt}\rho^{\eps}_-(t) =  \frac 1{2M} \Big( \rho^{\eps}(1,t) -\rho^{\eps}_-(t)\Big), \quad \frac{d}{dt}\rho^{\eps}_+(t) =  \frac 1{2M} \Big( \rho^{\eps}(N,t)-\rho^{\eps}_+(t)\Big)
 \end{equation}

Unless otherwise stated we fix hereafter $M=N$ and study the hydrodynamic limit
of $\rho^{\eps}(x,t)$ and $\rho^{\eps}_{\pm}(t)$ scaling space $x\to r$, $r=\eps x$ and time
$t \to \tau = \eps^2 t$.

\medskip
{\bf  The initial datum.}  Let $u_0 \in C(0,1)$ with values in $[0,1]$ and $v_{0,\pm}\in [0,1]$. Suppose that
\begin{equation}\label{e}
\rho^{\eps}_{\pm}(0) =  v_{0,\pm},\quad \rho^{\eps}(x,0) = u_0(\eps x), \forall x\in\Lambda_N.
 \end{equation}

In Section \ref{PrPDE} we will prove:

\medskip
\begin{thm}\label{PDE}
Let $N=M$ and $\rho^{\eps}_{\pm}(0) $ and
$\rho^{\eps}(x,0) $ as in \eqref{e}, then for any $r \in [0,1]$ and $t>0$,
$\rho^\eps([Nr],N^2t)$ and  $\rho^\eps_{\pm}(N^2t)$ have limits $u(r,t)$ and respectively $v_{\pm}(t)$.  $u(r,t)$ and $v_{\pm}(t)$ are the unique solutions of the ``free boundary problem'' stated below.

\end{thm}

\medskip
{\bf  The free boundary problem.}  
Find $u(r,t)$ and $v_{\pm}(t)$ with $u(r,t) \in C^{2,1}([0,1]\times (0,\infty))$ and $v_{\pm}(t)\in C^1(0,\infty)$ and continuous at $t=0$ so that
\begin{equation}
 \label{2.9}
 u_t =  \frac 1{2} u_{rr},  \quad u(1,t)=v_+(t),  u(0,t)=v_-(t),\quad \lim_{t\downarrow 0}u(r,t) =u_0(r)
 \end{equation}
with $v_{\pm}(t)$ such that $\dlim_{t\to 0}v_{\pm}(t) = v_{0,\pm}$ and 
\begin{equation}\label{2.10}
\begin{cases}
v_-(t)&=v_{0,-}+\dlim_{l\to 0}\dlim_{t_0\to 0}\dint_{t_0}^t \dfrac{1}{2}u_r(l,s)\,ds\\
v_+(t)&=v_{0,+}-\dlim_{l\to 1}\dlim_{t_0\to 0}\dint_{t_0}^t \dfrac{1}{2}u_r(l,s)\,ds.
\end{cases}
\end{equation}

For any given $v_{\pm}(t)$ (with the above regularity properties) the equation \eqref{2.9} is the linear heat equation with boundary values $v_{\pm}(t)$ and initial datum $u_0$. It has then a unique solution $u(r,t)$ with the desired regularity properties.  We call \eqref{2.9}--\eqref{2.10} a free boundary problem because the boundary data $v_{\pm}(t)$ are themselves unknown and must be determined so that \eqref{2.10} holds.

The above free boundary problem has been introduced by Strook and Williams, \cite{SW}, in the case where $[0,1]$ is replaced by the semi-infinite domain $[0,\infty)$, see also \cite{P}.  
\vskip.5cm

\subsection{Propagation of chaos}
Denote by $\mathbf{E}^{*,\eps}$ the expectation with respect to the process with generator \eqref{2.4}. Let $u_0 \in C(0,1)$ with values in $[0,1]$ and $n_{0,\pm}$ be fixed. Suppose that at the initial time, the particles in each reservoir are uniformly distributed and
\begin{equation}\label{2.8}
n_{\pm}(0) :=\sum_{z\in \mathcal S_{\pm}} \eta(z,0) =  n_{0,\pm},\quad \mathbf{E}^{*,\eps}[\eta(x,0)] = u_0(\eps x), \forall x\in\Lambda_N.
 \end{equation}
\begin{thm}\label{POC}
Suppose that $u_0\in C^1(0,1)$. For any distinct points $x_1, x_2$ in $\Lambda_N$ and any $t>0$,
\begin{equation*}
\dlim_{\eps\to 0} \Big|\mathbf{E}^{*,\eps}\big[\eta(x_1, \eps^{-2}t)\,\eta(x_2, \eps^{-2}t)\big]-\mathbf{E}^{*,\eps}[\eta(x_1,\eps^{-2}t)]\,\mathbf{E}^{*,\eps}[\eta(x_2,\eps^{-2}t)]\Big|=0.
\end{equation*}
\end{thm}
Theorem \ref{POC} is proved in Section \ref{PrPOC}.
\vskip.5cm
\subsection{Future works}
In a paper in preparation, we will show that the solution to the hydrodynamic limit equation can be represented probabilistically in terms of sticky Brownian motions. Moreover, we also study the long time behavior of our particle system at the time scale $N^{2+\alpha}t$ for some $\alpha>0$, where the size of each reservoir is much larger than the one of the channel.
\section{Proof of Theorem \ref{PDE}}\label{PrPDE}
We split the proof into several steps.
\subsection{Idea of the proof}
\begin{defin}\label{SRW}
Sticky random walk $(X(t))_{t\geq 0}$ moving on $\bar{\Lambda}_N$ is a continuous time random walk with jump rates $c(x,x\pm 1)=\dfrac{1}{2}, \forall x\in\Lambda_N$ and $c(0,1)=c(N+1,N)=\dfrac{1}{2N}$.

More precisely, the generator $L$ of the random walk $X$ is given by
\begin{align*}
Lf(x)=&\mathbf{1}_{1\leq x\leq N}\,\dfrac{1}{2}[f(x+1)+f(x-1)-2f(x)]\\
+&\mathbf{1}_{x=0}\,\dfrac{1}{2N}[f(x+1)-f(x)]+\mathbf{1}_{x=N+1}\,\dfrac{1}{2N}[f(x-1)-f(x)].
\end{align*}
\end{defin}

This means that when in $x \in \Lambda_N$, the sticky random walk $X$ waits an exponential time of mean $1$ and then jumps with equal probability to the neighboring sites. When it is at $0$ or $N+1$, the waiting time is exponential with mean $2N$ and the jump is respectively to $1$ and $N$.

We call $e^{Lt}, P_x, E_x$ and $p_t(\cdot,\cdot)$ the semigroup, the law, the expectation and the transition density of the random walk $X$ starting from $x$.
\begin{prop}
For any $x\in \bar{\Lambda}_N$ and $t\geq 0$,
\begin{equation}\label{duality*}
 \rho^{\eps}(x,t)=(e^{Lt}\rho_0^{\eps})(x)=E_x\big[\rho^{\eps}(X(t),0)\big]=\dsum_{y=0}^{N+1}\rho^{\eps}(y,0)\,p_t(x,y).
\end{equation}
\end{prop}
\begin{proof}
We claim that for $0\leq s<t$,
$$ \dfrac{d}{ds}E_x\Big[\rho^{\eps}(X(t-s),s)\Big] =0.$$
Indeed, by the definition of $\rho^{\eps}$,
$$ \dfrac{d}{ds}E_x\Big[\rho^{\eps}(X(t-s),s)\Big] = E_x \mathbf{E}^{\eps}\Big[\mathbf{L}\eta(X(t-s),s)\Big]-E_x \mathbf{E}^{\eps}\Big[L\eta(X(t-s),s)\Big].$$
Then the claim follows from checking that the actions of the two generators give the same expression. It implies that 
$$ \rho^{\eps}(x,t)=(e^{Lt}\rho_0^{\eps})(x)=E_x\big[\rho^{\eps}(X(t),0)\big]. $$
\end{proof}
Let $\tau$ be the first hitting time of $X$ into $\{0,N+1\}$.
\begin{cor}\label{stoptime}
$$ \rho^{\eps}(x,t)=E_x\Big[\rho^{\eps}(X(t),0)\mathbf{1}_{\tau>t}+\rho^{\eps}(X(\tau),t-\tau)\mathbf{1}_{\tau\leq t}\Big]. $$
\end{cor}
\begin{proof}
By using \eqref{duality*}, we write
\begin{align*}
\rho^{\eps}(x,t)=&\dsum_{y=0}^{N+1}\rho^{\eps}(y,0)P_x(X(t)=y)\\
=&\dsum_{y=0}^{N+1}\rho^{\eps}(y,0)\Big\{E_x\big[\mathbf{1}_{\tau>t}\mathbf{1}_{X(t)=y}\big]+E_x\big[\mathbf{1}_{\tau\leq t}E_{X(\tau)}[\mathbf{1}_{X(t-\tau)=y}]\big]\Big\}\\
=&E_x\Big[\mathbf{1}_{\tau>t}\dsum_{y=0}^{N+1}\rho^{\eps}(y,0)\mathbf{1}_{X(t)=y}\Big]+E_x\Big[\mathbf{1}_{\tau\leq t}E_{X(\tau)}\Big[\dsum_{y=0}^{N+1}\rho^{\eps}(y,0)\mathbf{1}_{X(t-\tau)=y}\Big]\Big]\\
=&E_x\Big[\rho^{\eps}(X(t),0)\mathbf{1}_{\tau>t}\Big]+E_x\Big[E_{X(\tau)}\big[\rho^{\eps}(X(t-\tau),0)\big]\mathbf{1}_{\tau\leq t}\Big]\\
=&E_x\Big[\rho^{\eps}(X(t),0)\mathbf{1}_{\tau>t}\Big]+E_x\Big[\rho^{\eps}(X(\tau),t-\tau)\mathbf{1}_{\tau\leq t}\Big].
\end{align*}
\end{proof}
\noindent
{\it Idea of the proof of Theorem \ref{PDE}.}
By Corollary \ref{stoptime}, we have a representation of $\rho^{\eps}$ in terms of the sticky random walk $X$ and $\rho^{\eps}_{\pm}$. Therefore, in order to obtain the convergence of $\rho^{\eps}$, we will first study about sticky random walks and derive some regularity properties of $\rho^{\eps}$. Due to these properties, we can attain that $\rho^{\eps}_{\pm}$ has a subsequence converging uniformly to some limit functions $v_{\pm}$, then establish a relation between $v_{\pm}$ and the limit function $u$ of the corresponding subsequence of $\rho^{\eps}$. Next we will identify $v_{\pm}$ in a weak form and prove their continuous differentiability in order to deduce the regularity of $u$. In the final step, the verification of the uniqueness of $v_{\pm}$ enables us to conclude the uniform convergence of the original sequence $\rho^{\eps}$.

\subsection{Sticky random walk}
Like the sticky Brownian motion can be realized as a time change of the standard Brownian motion, the sticky random walk can also be realized as a time change of the simple random walk. In the latter case, the analysis is quite elementary as we shall see. 

For $x=[Nr],r\in [0,1]$, let $Y^{\rm sp}$ be a simple symmetric random walk on $\mathbb{Z}$ starting from $x$. Recall that the sequence of rescaled random walks $N^{-1}Y^{\rm sp}(N^2t)$ converges uniformly almost surely on compact intervals of $[0,\infty)$ to a Brownian motion $B, B_0=r$, defined on some rich enough common probability space $(\tilde{\Omega}, \mathcal{F}, P)$, see \cite{K}.

We denote by $Y^{\rm rf}$ the simple random walk $Y^{\rm sp}$ reflected at $0$ and $N+1$. Namely, at $0$ and $N+1$, the reflecting random walk $Y^{\rm rf}$ jumps at rate $1$. The relation between reflecting and sticky random walk can be understood by looking at Figure 1 where each length of time spent by the reflecting random walk at $0$ and $N+1$ is amplified by a factor $2N-1$ for the sticky random walk. 
\begin{figure}
\begin{center}
\begin{tikzpicture}[scale=0.5]
\draw (0,0) -- (0,-4) [xshift=5cm] (0,0) -- (0,-1) -- (1,-1) -- (1,-2.7) -- (2,-2.7) -- (2,-3) -- (3,-3) -- (3,-4) [xshift=3cm] (0,0) -- (0,-4) [xshift=5cm] (0,0) -- (0,-4) [xshift=5cm] (0,0) -- (0,-1) -- (1,-1) -- (1,-2.7) -- (2,-2.7) -- (2,-3) -- (3,-3) -- (3,-3.5) [xshift=3cm] (0,0) -- (0,-4);
\draw (0,-4) -- (0,-14) [xshift=5cm] (3,-3.5)--(2,-3.5)--(2,-4.5)--(3,-4.5)--(3,-6)--(2,-6)--(2,-7.3)--(1,-7.3)--(1,-9.1) [xshift=3cm] (0,-4) -- (0,-14) [xshift=5cm] (0,-4)--(0,-14) [xshift=5cm] (3,-4.5)--(2,-4.5)--(2,-5.5)--(3,-5.5)--(3,-10)--(2,-10)--(2,-11.3)--(1,-11.3)--(1,-13.1) [xshift=3cm] (0,-4) -- (0,-14);
\draw[very thick] [xshift=5cm] (3,-3)--(3,-3.5) [xshift=13cm] (3,-3)--(3,-4.5);
\draw[very thick] [xshift=5cm] (3,-4.5)--(3,-6) [xshift=13cm] (3,-5.5)-- (3,-10);
\end{tikzpicture}
\end{center}
\caption{Reflecting random walk and sticky random walk on $\bar{\Lambda}_N$}
\end{figure}
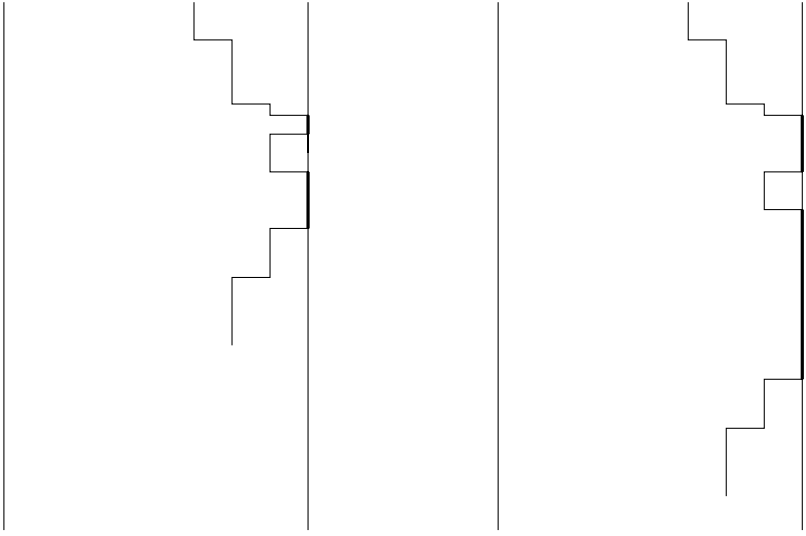

More precisely, let us call
  \begin{equation*}
 \mathbf{T}(0,N+1;t;Y^{\rm rf}) = \int_0^t \big( \mathbf 1_{Y^{\rm rf}(s)=0}+ \mathbf 1_{Y^{\rm rf}(s)=N+1}\big)\,ds
 \end{equation*}
 the local time spent by $Y^{\rm rf}$ at $0$ and $N+1$. Then 
\begin{prop}
The sticky random walk $X$ can be realized by setting
  \begin{equation}\label{realize}
X\Big(t+ (2N-1)\mathbf{T}(0,N+1;t;Y^{\rm rf})\Big)=Y^{\rm rf}(t).
 \end{equation}
\end{prop}
\begin{proof}
We will check that for this setting, $X$ is really a sticky random walk as already defined in Definition \ref{SRW}. More precisely, we show that $X$ stays at each site $x\in\Lambda_N$ for an exponential time of parameter $1$ and stays at each site $x\in\{0,N+1\}$ for an exponential time of parameter $1/(2N)$.

Firstly, suppose that $Y^{\rm rf}(t)=x\in\Lambda_N, \forall t_1\leq t< t_2$, where $Y^{\rm rf}(t_1^-)=x\pm 1$ and $Y^{\rm rf}(t_2)=x\pm 1$, so $P(t_2-t_1>\sigma)=e^{-\sigma}$. Then the length of time that $X$ spends at $x$ is given by
$$ \mathcal{T}_x(X):=t_2-t_1+(2N-1)\dint_{t_1}^{t_2}\big( \mathbf 1_{Y^{\rm rf}(s)=0}+ \mathbf 1_{Y^{\rm rf}(s)=N+1}\big)\,ds=t_2-t_1. $$

It is clear that $\mathcal{T}_x(X)\sim \Exp(1)$.

Secondly, suppose that $Y^{\rm rf}(t)=0, \forall s_1\leq t< s_2$, where $Y^{\rm rf}(s_1^-)=1$ and $Y^{\rm rf}(s_2)=1$, so $P(s_2-s_1>\sigma)=e^{-\sigma}$. Then the length of time that $X$ spends at $0$ is given by
\begin{align*}
\mathcal{T}_0(X):=&s_2-s_1+(2N-1)\dint_{s_1}^{s_2}\big( \mathbf 1_{Y^{\rm rf}(s)=0}+ \mathbf 1_{Y^{\rm rf}(s)=N+1}\big)\,ds\\
=&s_2-s_1+(2N-1)(s_2-s_1)=2N(s_2-s_1).
\end{align*}

It is easy to check that $\mathcal{T}_0(X)\sim \Exp(1/(2N))$. Similarly, we also obtain that $\mathcal{T}_{N+1}(X)\sim \Exp(1/(2N))$. It completes our verification.
\end{proof}
\begin{prop}
For any function $\varphi$ defined on $\bar{\Lambda}_N$,
\begin{equation}\label{selfadj}
 <\varphi, e^{Lt}\rho^{\eps}_0>=<\rho^{\eps}_0,e^{Lt}\varphi>, 
\end{equation}
where the product is defined by
$$ <f,g>=\dsum_{x=1}^Nf(x)g(x)+\dfrac{1}{\eps}\big(f(0)g(0)+f(N+1)g(N+1)\big). $$

Moreover, let $\mu^{\eps}$ be the probability on $\bar{\Lambda}_N$ defined by $c_{\eps}(\mathbf{1}_{x\in\Lambda_N}+\eps^{-1}\mathbf{1}_{x\in\{0,N+1\}})$, $c_{\eps}$ be the normalized constant. Then $\mu^{\eps}$ is invariant and reversible for the sticky process $X$. 
\end{prop}
\begin{proof}
To prove the equality \eqref{selfadj}, we need to verify the following equalities
\begin{align*}
p_t(x,y)&=p_t(y,x), \forall x,y\in\Lambda_N\\
p_t(0,N+1)&=p_t(N+1,0)\\
\dfrac{1}{\eps}p_t(0,x)&=p_t(x,0), \forall x\in\Lambda_N\\
\dfrac{1}{\eps}p_t(N+1,x)&=p_t(x,N+1), \forall x\in\Lambda_N.
\end{align*}

Since these equalities can be obtained by the same argument, we just check the third equality. Note that
$$p_t(0,x)=e^{Lt}\mathbf{1}_x(0)=\mathbf{1}_x(0)+\dsum_{n\geq 1}\dfrac{t^n}{n!}L^n\mathbf{1}_x(0).$$

By induction in $n$, we get $\dfrac{1}{\eps}L^n\mathbf{1}_x(0)=L^n\mathbf{1}_0(x), \forall x\in\Lambda_N$. Then the third equality follows. Now using the above equalities, we can rewrite the left hand side of \eqref{selfadj} in the following way
\begin{align*}
&\dsum_{x=1}^N\varphi(x)\Big[\dsum_{y=1}^Np_t(x,y)\rho^{\eps}(y,0)+p_t(x,0)\rho^{\eps}(0,0)+p_t(x,N+1)\rho^{\eps}(N+1,0)\Big]\\
&+\dfrac{1}{\eps}\varphi(0)\Big[\dsum_{y=1}^Np_t(0,y)\rho^{\eps}(y,0)+p_t(0,0)\rho^{\eps}(0,0)+p_t(0,N+1)\rho^{\eps}(N+1,0)\Big]\\
&+\dfrac{1}{\eps}\varphi(N+1)\Big[\dsum_{y=1}^Np_t(N+1,y)\rho^{\eps}(y,0)+p_t(N+1,0)\rho^{\eps}(0,0)\\
&\quad\quad\hskip5.5cm+p_t(N+1,N+1)\rho^{\eps}(N+1,0)\Big]\\
=&\dsum_{y=1}^N\rho^{\eps}(y,0)\Big[\dsum_{x=1}^Np_t(y,x)\varphi(x)+p_t(y,0)\varphi(0)+p_t(y,N+1)\varphi(N+1)\Big]\\
&+\dfrac{1}{\eps}\rho^{\eps}(0,0)\Big[\dsum_{x=1}^Np_t(0,x)\varphi(x)+p_t(0,0)\varphi(0)+p_t(0,N+1)\varphi(N+1)\Big]\\
&+\dfrac{1}{\eps}\rho^{\eps}(N+1,0)\Big[\dsum_{x=1}^Np_t(N+1,x)\varphi(x)+p_t(N+1,0)\varphi(0)\\
&\quad\quad\hskip5.5cm+p_t(N+1,N+1)\varphi(N+1)\Big]\\
\end{align*}
and it is equal to the right hand side of \eqref{selfadj}. 

Moreover, by a similar argument as above, it is easy to check that for any $t\geq 0$, $\mu^{\eps}e^{Lt}=\mu^{\eps}$ and $\mu^{\eps}$ satisfies the detailed balance condition $\mu^{\eps}(x)c(x,y)=\mu^{\eps}(y)c(y,x), \forall x,y\in\bar{\Lambda}_N$. It implies that $\mu^{\eps}$ is invariant and reversible for the process $X$.
\end{proof}
\begin{cor}[Conservation of mass]
$$ \dsum_{x=1}^N\rho^{\eps}_0(x)+\dfrac{1}{\eps}\big(\rho_0^{\eps}(0)+\rho^{\eps}_0(N+1)\big)=  \dsum_{x=1}^N\rho^{\eps}(x,t)+\dfrac{1}{\eps}\big(\rho^{\eps}(0,t)+\rho^{\eps}(N+1,t)\big), \forall t\geq 0.$$
\end{cor}
\begin{proof}
The result follows from taking $\varphi\equiv 1$ in \eqref{selfadj} and noting that for any $x\in \bar{\Lambda}_N$, $\rho^{\eps}(x,t)=e^{Lt}\rho_0^{\eps}(x)$. 
\end{proof}
\subsection{Regularity properties of $\rho^{\eps}(0,t), \rho^{\eps}(N+1,t)$}
Let us call $\tau_a^{Z}=\inf\{s\geq 0, Z(s)=a\}$ the first time that a process $Z$ hits $a$.
\begin{prop}\label{uniconv}
There exists a constant $C$ such that for any $s,t\geq 0$,
$$|\rho^{\eps}(0,t)-\rho^{\eps}(0,s)|\leq C\eps|t-s|^{\frac{1}{2}}$$
$$|\rho^{\eps}(N+1,t)-\rho^{\eps}(N+1,s)|\leq C\eps|t-s|^{\frac{1}{2}}.$$
\end{prop}
\begin{proof}
We will prove the first inequality. For the second inequality, we use a similar argument. From \eqref{duality*}, we can write
\begin{align*}
\rho^{\eps}(0,t) =&\rho^{\eps}(0,0)p_t(0,0)+\eps\dsum_{x=1}^N\rho^{\eps}(x,0)p_t(x,0)+\rho^{\eps}(N+1,0)p_t(N+1,0)\\
\leq& \rho^{\eps}(0,0)+\eps\dsum_{x=1}^Np_t(x,0)+p_t(N+1,0).
\end{align*}

Then we need to bound $p_t(x,0), \forall x\in\Lambda_N$ and $p_t(N+1,0)$. Firstly, we have
$$p_t(x,0)=\dint_0^tP_0(X(t-s)=0)P_x(\tau_0^X=s)\,ds\leq P_x(\tau_0^X\leq t).$$

We consider two cases as follows. For the first case when the sticky random walk $X$ reaches $0$ before $N+1$, we obtain that
$$P_x(\tau_0^X\leq t, \tau_0^X<\tau^X_{N+1})\leq P_x(\tau_0^{Y^{\rm sp}}\leq t)\leq 2P_x(Y^{\rm sp}(t)\leq 0).$$

For the second case when the sticky random walk $X$ reaches $N+1$ before $0$, we observe that by \eqref{realize}, the sticky random walk $X$ always reaches $0$ after the corresponding reflecting random walk $Y^{\rm rf}$. Hence, 
\begin{align*}
P_x(\tau_0^X\leq t, \tau_0^X>\tau^X_{N+1})\leq &P_x(\tau_0^{Y^{\rm rf}}\leq t)\leq P_x(\tau_0^{Y^{\rm sp}}\leq t)+P_x(\tau_{2(N+1)}^{Y^{\rm sp}}\leq t) \\
\leq &2P_x(Y^{\rm sp}(t)\leq 0)+2P_x(Y^{\rm sp}(t)\geq 2(N+1)).
\end{align*}

Then 
$$ \dsum_{x=1}^Np_t(x,0)\leq 4\dsum_{x=1}^N P_x(Y^{\rm sp}(t)\leq 0)+2\dsum_{x=1}^N P_x(Y^{\rm sp}(t)\geq 2(N+1)).$$

By the local central limit theorem, \cite{LL}, there exist positive constants $c_1, \ldots, c_5$ such that
\begin{align}
&|P_x(Y^{\rm sp}(t)=y)-G_t(x,y)|\leq \dfrac{c_1}{\sqrt{t}}G_t(x,y), \quad |x-y|\leq t^{5/8}\notag\\
&P_x(Y^{\rm sp}(t)=y)\leq \min\big(c_2e^{-c_3|x-y|^2/t},c_4e^{-|x-y|(\log |x-y|-c_5)}\big), \quad |x-y|>t^{5/8}\label{LL},
\end{align}
where 
$$ G_t(x,y)=\dfrac{1}{\sqrt{2\pi t}}e^{-\frac{(x-y)^2}{2t}}. $$

After some computations, we obtain that there exists a universal constant $C$ such that
$$ \dsum_{x=1}^N\dsum_{y\leq 0}P_x(Y^{\rm sp}(t)=y)\leq C\sqrt{t}, \quad\quad\quad \dsum_{x=1}^N\dsum_{y\geq 2(N+1)}P_x(Y^{\rm sp}(t)=y)\leq C\sqrt{t}. $$

It implies that 
$$ \dsum_{x=1}^Np_t(x,0)\leq C\sqrt{t}. $$

Secondly, we bound $p_t(N+1,0)$ by using the same argument as above. Namely,
\begin{align*}
p_t(N+1,0)&=\dint_0^tP_0(X(t-s)=0)P_{N+1}(\tau_0^X=s)\,ds\leq P_{N+1}(\tau_0^X\leq t)\\
&\leq P_{N+1}(\tau_0^{Y^{\rm rf}}\leq t)\leq P_{N+1}(\tau_0^{Y^{\rm sp}}\leq t)+P_{N+1}(\tau_{2(N+1)}^{Y^{\rm sp}}\leq t) \\
&\leq 2P_{N+1}(Y^{\rm sp}(t)\leq 0)+2P_{N+1}(Y^{\rm sp}(t)\geq 2(N+1)).
\end{align*}

Using \eqref{LL} again gives us
$$ \dsum_{y\leq 0}P_{N+1}(Y^{\rm sp}(t)=y)\leq C\eps\sqrt{t}, \quad\quad\quad \dsum_{y\geq 2(N+1)}P_{N+1}(Y^{\rm sp}(t)=y)\leq C\eps\sqrt{t}, $$
for some constant $C$.

Hence, we have just proved that $\rho^{\eps}(0,t)-\rho^{\eps}(0,0)\leq C\eps\sqrt{t}$. It remains to check that $\rho^{\eps}(0,0)-\rho^{\eps}(0,t)\leq C\eps\sqrt{t}$. To do this, we notice that
\begin{align*}
\rho^{\eps}(0,0)=&\rho^{\eps}(0,0)\Big(\dsum_{x=0}^{N+1}p_t(0,x)\Big)\\
=&\rho^{\eps}(0,0)p_t(0,0)+\eps\dsum_{x=1}^N\rho^{\eps}(0,0)p_t(x,0)+\rho^{\eps}(0,0)p_t(N+1,0)\\
\leq &\rho^{\eps}(0,t)+\eps\dsum_{x=1}^Np_t(x,0)+p_t(N+1,0),
\end{align*}
where the last inequality follows from the fact that
$$\rho^{\eps}(0,t)=\rho^{\eps}(0,0)p_t(0,0)+\dsum_{x=1}^{N+1}\rho^{\eps}(x,0)p_t(0,x).$$

Then by a similar argument, we obtain the desired inequality.
\end{proof}

Let us call $\tilde{\rho}_{\pm}^{\eps}(s):=\rho^{\eps}_{\pm}(\eps^{-2}s)$. Then the following result follows from Proposition \ref{uniconv}.
\begin{cor}\label{subseq}
There exist subsequences $\tilde{\rho}_{\pm}^{\eps_k}$ that converge uniformly to $v_{\pm}$, respectively.
\end{cor}
\subsection{Convergence by subsequences}\label{rw}
We denote $\tilde{\rho}^{\eps}(r,t):=\rho^{\eps}([\eps^{-1}r],\eps^{-2}t)$.
\begin{thm}
For any $r\in [0,1], T>0$, the subsequence $ \tilde{\rho}^{\eps_k}$ converges to $u$ uniformly on $[0,T]$. The hydrodynamic limit $u$ is defined, for $r\in(0,1)$ and $t>0$, by
\begin{align}
 u(r,t)=&\dint_0^1u_0(r')[\Theta(r-r',t)-\Theta(r+r',t)]\,dr'\notag\\
&-\dint_{[0,t)}\dfrac{\partial \Theta}{\partial r}(r,t-s)v_-(s)\,ds+\dint_{[0,t)}\dfrac{\partial \Theta}{\partial r}(r-1,t-s)v_+(s)\,ds,\label{exu}
\end{align}
where 
$$ \Theta(r,t)=\dsum_{n=-\infty}^{+\infty}\dfrac{1}{\sqrt{2\pi t}}e^{-\frac{(r+2n)^2}{2t}}. $$
\end{thm}
\begin{proof}
We set
\begin{align*}
\mathbf{F}_{N,x}(s)&=P_x(\tau^X_0\leq s, \tau^X_0<\tau^X_{N+1})\\
\mathbf{G}_{N,x}(s)&=P_x(\tau^X_{N+1}\leq s, \tau^X_0>\tau^X_{N+1}).
\end{align*}

For $x=[\eps^{-1}r]$, let $Y^{\rm sp}$ be a simple symmetric random walk on $\mathbb{Z}$ starting from $x$. By Corollary \ref{stoptime}, we can write 
\begin{align*}
&\tilde{\rho}^{\eps}(r,t)=E_x\big[u_0(\eps X(\eps^{-2}t))\mathbf{1}_{\tau>\eps^{-2}t}\big]+E_x\big[\rho^{\eps}(X(\tau),\eps^{-2}t-\tau)\mathbf{1}_{\tau\leq \eps^{-2}t}\big]\\
=&E_x\big[u_0(\eps Y^{\rm sp}(\eps^{-2}t))\mathbf{1}_{\tau_0^{Y^{\rm sp}}\wedge\tau_{N+1}^{Y^{\rm sp}} >\eps^{-2}t}\big]+\dint_0^{\eps^{-2}t}\rho^{\eps}_-(\eps^{-2}t-s)\,d\mathbf{F}_{N,x}(s)\\
&\hskip6.5cm+\dint_0^{\eps^{-2}t}\rho^{\eps}_+(\eps^{-2}t-s)\,d\mathbf{G}_{N,x}(s)\\
=&E_x\big[u_0(\eps Y^{\rm sp}(\eps^{-2}t))\mathbf{1}_{\tau_0^{Y^{\rm sp}}\wedge\tau_{N+1}^{Y^{\rm sp}} >\eps^{-2}t}\big]-\dint_0^t\rho_-^{\eps}(\eps^{-2}s)d\mathbf{F}_{N,x}(\eps^{-2}(t-s))\\
&\hskip6.5cm+\dint_0^t\rho_+^{\eps}(\eps^{-2}s)d\mathbf{G}_{N,x}(\eps^{-2}(t-s)).
\end{align*}

Notice that for $x\in \bar{\Lambda}_N$, 
$$ \mathbf{F}_{N,x}(s)=P_x(\tau^{Y^{\rm sp}}_0\leq s, \tau^{Y^{\rm sp}}_0<\tau^{Y^{\rm sp}}_{N+1})=:\mathds{F}_{N,x}(s), $$
$$ \mathbf{G}_{N,x}(s)=P_x(\tau^{Y^{\rm sp}}_{N+1}\leq s, \tau^{Y^{\rm sp}}_0>\tau^{Y^{\rm sp}}_{N+1})=:\mathds{G}_{N,x}(s). $$

Recall that the sequence of rescaled random walks $\eps Y^{\rm sp}(\eps^{-2}t)$ converges uniformly in distribution on compact intervals of $[0,\infty)$ to a Brownian motion $B, B_0=r\in [0,1]$, defined on some rich enough common probability space $(\tilde{\Omega}, \mathcal{F}, P)$. Moreover, it can be shown that the sequences $\eps^2\tau_0^{Y^{\rm sp}}$ and $\eps^2\tau_{N+1}^{Y^{\rm sp}}$ converge to $\tau_0^B$ and $\tau_1^B$, respectively, in distribution. Then the corresponding sequences of distributions $\mathds{F}_{N,x}(\eps^{-2}s), \mathds{G}_{N,x}(\eps^{-2}s)$ converge uniformly to $F_r(s), G_r(s)$, respectively, as $\eps\to 0$, where 
\begin{align*}
F_r(s)&=P_r(\tau^B_0\leq s, \tau^B_0<\tau^B_1)\\
G_r(s)&=P_r(\tau^B_1\leq s, \tau^B_0>\tau^B_1).
\end{align*}

On the other hand, by Corollary \ref{subseq}, there exist subsequences $\tilde{\rho}_{\pm}^{\eps_k}$ that converge uniformly to $v_{\pm}$ respectively. From this uniform convergence and Proposition \ref{uniconv}, we also deduce that $v_{\pm}\in C([0,\infty))$. 

Therefore, the subsequence $\tilde{\rho}^{\eps_k}$ converges to $u$ uniformly on $[0,T]$, for any $r\in [0,1],T>0$, where the hydrodynamic limit $u$ is given by
\begin{align}
 u(r,t)=&\dint_0^1 u_0(r')\,P_r(B(t)=r', \tau^B_0\wedge \tau_1^B>t)\,dr'\notag\\
&-\dint_0^t v_-(s)dF_r(t-s)+\dint_0^t v_+(s)dG_r(t-s).\label{u}
\end{align}

For $r\in(0,1)$ and $t>0$,
\begin{align}
 u(r,t)=&\dint_0^1u_0(r')[\Theta(r-r',t)-\Theta(r+r',t)]\,dr'\notag\\
&-\dint_{[0,t)}\dfrac{\partial \Theta}{\partial r}(r,t-s)v_-(s)\,ds+\dint_{[0,t)}\dfrac{\partial \Theta}{\partial r}(r-1,t-s)v_+(s)\,ds.\label{exu}
\end{align}
\end{proof}

It can be directly verified that $u\in C^{2,1}((0,1)\times (0,\infty))$. Moreover, the fact that $u$ satisfies the linear heat equation \eqref{2.9} with boundary values $v_{\pm}$ and initial datum $u_0$ is already shown in Chapter 6 of \cite{C}.
\subsection{Identification of the limit functions $v_{\pm}$}\label{bdrylimit}
The conservation of mass implies that the flux at $0$ and $1$ equals the density change in the left and right reservoir respectively, we thus expected that
\begin{equation}\label{dot}
\dfrac{1}{2}u_r(0,t)=\dfrac{d}{dt}v_-(t), \quad\quad\quad -\dfrac{1}{2}u_r(1,t)=\dfrac{d}{dt}v_+(t).
\end{equation}

We will eventually prove \eqref{dot}, but this will require several intermediate steps because at this point we do not know whether the left and right hand sides of \eqref{dot} are well-defined. In Proposition \ref{3.8} below, we will show \eqref{dot} in a weak form.
\begin{prop}\label{3.8}
For any $l\in [0,1]$ and $t,t_0\in (0,\infty)$,
\begin{align}
\dint_0^l u(r,t)\,dr-\dint_0^l u(r,t_0)\,dr&=\dint_{t_0}^t \dfrac{1}{2}u_r(l,s)\,ds+v_-(t_0)-v_-(t)\label{left}\\
\dint_l^1 u(r,t)\,dr-\dint_l^1 u(r,t_0)\,dr&=\dint_{t_0}^t -\dfrac{1}{2}u_r(l,s)\,ds+v_+(t_0)-v_+(t).\label{right}
\end{align}
\end{prop}
\begin{proof}
Call $L=\eps^{-1}l, H=\eps^{-1}h$. We consider the average
$$ \psi_{\eps,h}(l,t):=\dfrac{1}{H}\dsum_{x=L-H+1}^L\,\eps\dsum_{y=1}^x\,\rho^{\eps}(y,\eps^{-2}t). $$

Let us recall the differential system \eqref{2.6}, \eqref{2.7}, then for any $t,t_0>0$, the difference $\psi_{\eps,h}(l,t)-\psi_{\eps,h}(l,t_0)$ can be written as follows.
\begin{align*}
&\psi_{\eps,h}(l,t)-\psi_{\eps,h}(l,t_0)\\
=&\dfrac{1}{H}\dsum_{x=L-H+1}^L\eps\dint_{\eps^{-2}t_0}^{\eps^{-2}t}\dsum_{y=1}^x\dfrac{d}{ds}\rho^{\eps}(y,s)\,ds\\
=&\dfrac{1}{H}\dsum_{x=L-H+1}^L\eps\dint_{t_0}^t\eps^{-2}\dfrac{1}{2}\big[\rho^{\eps}(x+1,\eps^{-2}s)-\rho^{\eps}(x,\eps^{-2}s)+\rho^{\eps}(0,\eps^{-2}s)-\rho^{\eps}(1,\eps^{-2}s)\big]\,ds\\
=&\dint_{t_0}^t\dfrac{1}{2}\dfrac{\rho^{\eps}(L+1,\eps^{-2}s)-\rho^{\eps}(L-H+1,\eps^{-2}s)}{h}\,ds+\rho^{\eps}_-(\eps^{-2}t_0)-\rho^{\eps}_-(\eps^{-2}t).
\end{align*}

Taking the limit along the subsequence $\rho^{\eps_k}$ and then $h\to 0$ of the right hand side of the last equality will give us the right hand side of \eqref{left}. On the other hand, since 
\begin{align*}
\psi_{\eps,h}(l,t)&=\dfrac{\eps}{H}\bigg(H\dsum_{x=1}^{L-H+1}\rho^{\eps}(x,\eps^{-2}t)+\dsum_{y=L-H+2}^L\,(L+1-y)\rho^{\eps}(y,\eps^{-2}t)\bigg)\\
&\leq \eps\dsum_{x=1}^{L-H+1}\rho^{\eps}(x,\eps^{-2}t)+\dfrac{\eps}{H}\dsum_{y=L-H+2}^L\,(L+1-y)\\
&\leq \eps\dsum_{x=1}^{L-H+1}\rho^{\eps}(x,\eps^{-2}t)+\dfrac{h-\eps}{2},
\end{align*}
then along the subsequence $\rho^{\eps_k}$ we obtain
$$ \lim_{h\to 0}\lim_{\eps\to 0}\psi_{\eps,h}(l,t)-\psi_{\eps,h}(l,t_0)=\dint_0^lu(r,t)\,dr-\dint_0^lu(r,t_0)\,dr. $$

It implies \eqref{left}. To derive \eqref{right}, we apply the same argument as before by considering the average
$$ \phi_{\eps,h}(l,t):=\dfrac{1}{H}\dsum_{x=L+1}^{L+H}\,\eps\dsum_{y=x+1}^N\,\rho^{\eps}(y,\eps^{-2}t). $$
\end{proof}

Now taking the limit $t_0\to 0, l\to 0$ in \eqref{left} and $t_0\to 0, l\to 1$ in \eqref{right} gives us the boundary conditions
\begin{equation}\label{bdry}
\begin{cases}
v_-(t)&=v_{0,-}+\dlim_{l\to 0}\dlim_{t_0\to 0}\dint_{t_0}^t \dfrac{1}{2}u_r(l,s)\,ds\\
v_+(t)&=v_{0,+}-\dlim_{l\to 1}\dlim_{t_0\to 0}\dint_{t_0}^t \dfrac{1}{2}u_r(l,s)\,ds.
\end{cases}
\end{equation}

Since $\dfrac{\partial^2\Theta}{\partial r^2}=-2\dfrac{\partial\Theta}{\partial s}$, it implies that
\begin{align*}
u_r(r,t)=&\dint_0^1u_0(r')\Bigg[\dfrac{\partial\Theta}{\partial r}(r-r',t)-\dfrac{\partial\Theta}{\partial r}(r+r',t)\Bigg]dr'\\
&+\dint_{[0,t)}2\dfrac{\partial \Theta}{\partial s}(r,t-s)v_-(s)\,ds-\dint_{[0,t)}2\dfrac{\partial \Theta}{\partial s}(r-1,t-s)v_+(s)\,ds.
\end{align*}

Then
\begin{align*}
 \dint_{t_0}^t\dfrac{1}{2}u_r(l,s)\,ds=&\dint_{t_0}^t\dfrac{1}{2}\dint_0^1u_0(r')\Bigg[\dfrac{\partial\Theta}{\partial r}(l-r',s)-\dfrac{\partial\Theta}{\partial r}(l+r',s)\Bigg]dr'ds\\
&-\dint_{[0,t)}\Theta(l,t-s)v_-(s)\,ds+\dint_0^{t_0}\Theta(l,t_0-s)v_-(s)\,ds\\
&+\dint_{[0,t)}\Theta(l-1,t-s)v_+(s)\,ds-\dint_0^{t_0}\Theta(l-1,t_0-s)v_+(s)\,ds.
\end{align*}

Hence, by \eqref{bdry}, it follows that
\begin{align}
v_-(t)=v_{0,-}+&\dlim_{l\to 0}\dlim_{t_0\to 0}\dint_{t_0}^t\dfrac{1}{2}\dint_0^1u_0(r')\Bigg[\dfrac{\partial\Theta}{\partial r}(l-r',s)-\dfrac{\partial\Theta}{\partial r}(l+r',s)\Bigg]dr'ds\notag\\
&-\dint_{[0,t)}\Theta(0,t-s)v_-(s)\,ds+\dint_{[0,t)}\Theta(1,t-s)v_+(s)\,ds,\label{first}
\end{align}
and similarly,
\begin{align}
v_+(t)=v_{0,+}-&\dlim_{l\to 1}\dlim_{t_0\to 0}\dint_{t_0}^t\dfrac{1}{2}\dint_0^1u_0(r')\Bigg[\dfrac{\partial\Theta}{\partial r}(l-r',s)-\dfrac{\partial\Theta}{\partial r}(l+r',s)\Bigg]dr'ds\notag\\
&+\dint_{[0,t)}\Theta(1,t-s)v_-(s)\,ds-\dint_{[0,t)}\Theta(0,t-s)v_+(s)\,ds.\label{second}
\end{align}
\subsection{Continuous differentiability of $v_{\pm}$}
The idea to prove the continuous differentiability of $v_{\pm}$ comes from \cite{MF}. 
We first set
$$f_{\pm}(t)=v_{0,\pm}\mp\dlim_{l\to 0}\dlim_{t_0\to 0}\dint_{t_0}^t\dfrac{1}{2}\dint_0^1u_0(r')\Bigg[\dfrac{\partial\Theta}{\partial r}(l-r',s)-\dfrac{\partial\Theta}{\partial r}(l+r',s)\Bigg]dr'ds.$$

It is easy to obtain that
\begin{align*}
f'_-(t)=&-\dint_0^1u_0(r')\dfrac{\partial\Theta}{\partial r'}(r',t)dr'\\
f'_+(t)=&\dint_0^1u_0(r')\dfrac{\partial\Theta}{\partial r'}(1-r',t)dr'.
\end{align*}

Let us consider the equation
\begin{align}\label{Vol}
\begin{bmatrix}
V_-(t)\\ 
V_+(t)\\ 
\end{bmatrix}=&\begin{bmatrix}
f_-'(t)+K_-(t)v_{0,-}+K_+(t)v_{0,+}\\ 
f_+'(t)+K_+(t)v_{0,-}+K_-(t)v_{0,+}\\ 
\end{bmatrix}\notag\\
&\quad\quad+\dint_0^t \begin{bmatrix}
K_-(t-s)&K_+(t-s)\\ 
K_+(t-s)&K_-(t-s)\\ 
\end{bmatrix}\begin{bmatrix}
V_-(s)\\ 
V_+(s)\\ 
\end{bmatrix}ds,
\end{align}
where $K_-(t)=-\Theta(0,t), K_+(t)=\Theta(1,t)$.
\begin{prop}\label{Vol1}
The equation \eqref{Vol} has a unique solution $(V_-,V_+)\in C(0,T]\times C(0,T]$ for any $T>0$.
\end{prop}
\begin{proof}
For any $T>0$, let us fix $\delta\in(0,T)$ and set $\mathds{V}_{\pm}(t)=V_{\pm}(t+\delta)$. We claim that $\mathds{V}_{\pm}\in C[0,T-\delta]$. Indeed, we can easily check that the functions $f'_{\pm}, K_{\pm}$ belong to $C[\delta,T]\cap L^1(0,T)$. Moreover,
\begin{align}
V_{\pm}(t+\delta)=&f'_{\pm}(t+\delta)+K_{\pm}(t+\delta)v_{0,-}+K_{\mp}(t+\delta)v_{0,+}+J_{\pm}(t)\notag\\
&+\dint_0^t\big[K_{\pm}(t-s)V_-(s+\delta)+K_{\mp}(t-s)V_+(s+\delta)\big]\,ds,\label{third}
\end{align}
where $J_{\pm}(t)=\dint_0^{\delta}\big[K_{\pm}(t+\delta-s)V_-(s)+K_{\mp}(t+\delta-s)V_+(s)\big]\,ds$ are continuous in $t\in[0,T-\delta]$.

 Then \eqref{third} can be rewritten as
$$ \mathds{V}_{\pm}(t)=\mathds{F}_{\pm}(t)+\dint_0^t\big[K_{\pm}(t-s)\mathds{V}_-(s)+K_{\mp}(t-s)\mathds{V}_+(s)\big]\,ds $$
where $\mathds{F}_{\pm}$ are continuous functions on $[0,T-\delta]$.

Let us denote $\mathds{V}(t)=(\mathds{V}_-(t),\mathds{V}_+(t))$ and define the operator $T$ by
$$ T\mathds{V}(t):=\begin{bmatrix}
\mathds{F}_-(t)\\ 
\mathds{F}_+(t)\\ 
\end{bmatrix}+\dint_0^t\begin{bmatrix}
K_-(t-s)&K_+(t-s)\\ 
K_+(t-s)&K_-(t-s)\\ 
\end{bmatrix}\mathds{V}(s)\,ds. $$

Since $K_{\pm}\in L^1(0,T)$, it allows us to choose an integer $I$ such that if $S=T/I$ then 
$$ \dint_0^S|K_{\pm}(t)|\,dt=\alpha_{\pm}<\dfrac{1}{2\sqrt{2}}. $$

If $\mathds{V}\in C[0,S]\times C[0,S]$, we can verify $T\mathds{V}\in C[0,S]\times C[0,S]$ by noting the fact that $\mathds{F}_{\pm}\in C[0,T-\delta]$ and $K_{\pm}\in C[\delta,T]$.

Moreover, $T$ is a contraction mapping on $C[0,S]\times C[0,S]$. Hence, by the contraction mapping theorem, there exists a unique $\mathds{V}\in C[0,S]\times C[0,S]$ such that $T\mathds{V}=\mathds{V}$.  

Next, let us define $\mathds{V}^{(1)}(t)=(\mathds{V}^{(1)}_-(t),\mathds{V}^{(1)}_+(t))$, where $\mathds{V}_{\pm}^{(1)}(t)=\mathds{V}_{\pm}(t+S)$. Then
$$ \mathds{V}^{(1)}_{\pm}(t)=\mathds{F}^{(1)}_{\pm}(t)+\dint_0^t\big[K_{\pm}(t-s)\mathds{V}^{(1)}_-(s)+K_{\mp}(t-s)\mathds{V}^{(1)}_+(s)\big]\,ds, $$
where
$$ \mathds{F}^{(1)}_{\pm}(t)=\mathds{F}_{\pm}(t+S)+\dint_0^S\big[K_{\pm}(t+S-s)\mathds{V}_-(s)+K_{\mp}(t+S-s)\mathds{V}_+(s)\big]\,ds. $$

Then the contraction mapping argument as before enables us to attain that $\mathds{V}^{(1)}\in C[0,S]\times C[0,S]$. By induction on the intervals $(iS, (i+1)S)$, we obtain that $\mathds{V}_{\pm}\in C[0,T-\delta]$. This completes the proof of our beginning claim.

 Since $\delta>0$ is arbitrary, the claim implies that $V_{\pm}\in C(0,T]$.
\end{proof}
\begin{prop}\label{C1}
The solutions $v_{\pm}$ to the system \eqref{first} and \eqref{second} belong to $C^1(0,\infty)$. Moreover, for any $T>0$, $v'_{\pm}(t)=V_{\pm}(t), \forall t\in (0,T]$, where $(V_-,V_+)$ is the unique solution to \eqref{Vol}.
\end{prop}
\begin{proof}
For any $T>0$, we fix any $\delta\in (0,T/2)$ and define
$$ Q_{\pm}(t,h)=\dfrac{v_{\pm}(t+h)-v_{\pm}(t)}{h}, $$
for $h\in (0,\delta], t\in (0,T-\delta]$. Since $v_{\pm}$ solve \eqref{first} and \eqref{second}, then $Q_{\pm}$ satisfy
$$Q_{\pm}(t,h)=\mathcal{G}_{\pm}(t,h)+\dint_0^t[K_{\pm}(t-s)Q_-(s,h)+K_{\mp}(t-s)Q_+(s,h)]\,ds$$
where 
$$\mathcal{G}_{\pm}(t,h)=\dfrac{f_{\pm}(t+h)-f_{\pm}(t)}{h}+\dfrac{1}{h}\dint_t^{t+h}[K_{\pm}(s)v_-(t+h-s)+K_{\mp}(s)v_+(t+h-s)]\,ds.$$

Set $\mathcal{F}_{\pm}(t)=f_{\pm}'(t)+K_{\pm}(t)v_{0,-}+K_{\mp}(t)v_{0,+}$. Therefore, the differences $D_{\pm}(t,h)=Q_{\pm}(t,h)-V_{\pm}(t)$ satisfy
$$D_{\pm}(t,h)=\mathcal{G}_{\pm}(t,h)-\mathcal{F}_{\pm}(t)+\dint_0^t[K_{\pm}(t-s)D_-(s,h)+K_{\mp}(t-s)D_+(s,h)]\,ds.$$

Then
\begin{align*}
&|D_-(t,h)|+|D_+(t,h)|\\
\leq &|\mathcal{G}_-(t+h)-\mathcal{F}_-(t)|+|\mathcal{G}_+(t+h)-\mathcal{F}_+(t)|\\
&+\dint_0^t\big[|K_-(t-s)|+|K_+(t-s)|\big]\,\big[|D_-(s,h)|+|D_+(s,h)|\big]\,ds. 
\end{align*}

Call $q(t,h)$ the nonnegative function such that
\begin{align*}
&|D_-(t,h)|+|D_+(t,h)|\\
=&\mathcal{E}(t,h)+\dint_0^t\big[|K_-(t-s)|+|K_+(t-s)|\big]\,\big[|D_-(s,h)|+|D_+(s,h)|\big]\,ds, 
\end{align*}
where $\mathcal{E}(t,h):=|\mathcal{G}_-(t+h)-\mathcal{F}_-(t)|+|\mathcal{G}_+(t+h)-\mathcal{F}_+(t)|-q(t,h)$.

Then $|D_-(t,h)|+|D_+(t,h)|$ can be represented as
$$|D_-(t,h)|+|D_+(t,h)|=\mathcal{E}(t,h)+\dint_0^t\mathcal{R}(t-s)\mathcal{E}(s,h)\,ds,$$
where the resolvent $\mathcal{R}(t)$ associated with $|K_-(t)|+|K_+(t)|$ is defined as the unique solution to the linear equation
$$ \mathcal{R}(t)=|K_-(t)|+|K_+(t)|+\dint_0^t\big[|K_-(t-s)|+|K_+(t-s)|\big]\mathcal{R}(s)\,ds. $$

Since $|K_-|+|K_+|\in L^1(0,T)$, then by Lemma 1 in \cite{MF}, we get that $\mathcal{R}\in L^1(0,T)$ and $\mathcal{R}\geq 0$ a.e. Therefore, 
\begin{align*}
&|D_-(t,h)|+|D_+(t,h)|\\
\leq & |\mathcal{G}_-(t+h)-\mathcal{F}_-(t)|+|\mathcal{G}_+(t+h)-\mathcal{F}_+(t)|\\
&+\dint_0^t\mathcal{R}(t-s)[|\mathcal{G}_-(s+h)-\mathcal{F}_-(s)|+|\mathcal{G}_+(s+h)-\mathcal{F}_+(s)|]\,ds. 
\end{align*}

Since $ \dlim_{h\downarrow 0} \mathcal{G}_{\pm}(t,h)=\mathcal{F}_{\pm}(t) $ and $\mathcal{R}\in L^1(0,T)$, we obtain
$$ \dlim_{h\downarrow 0}|D_-(t,h)|+|D_+(t,h)|=0. $$

It implies that $Q_{\pm}(t,h)\to V_{\pm}(t)$ as $h\downarrow 0$ uniformly in $t\in[\delta,T-\delta]$.

Since the convergence is uniform in $t\in [\delta,T-\delta]$, it follows that the sequences $\{Q_{\pm}(\cdot,h)\}_{h\in(0,\delta)}$ in $C[\delta,T-\delta]$ are equicontinuous. Hence,
$$ \dlim_{h\uparrow 0} Q_{\pm}(t,h)=\dlim_{h\downarrow 0}Q_{\pm}(t-h,h)=\dlim_{h\downarrow 0}Q_{\pm}(t,h)=V_{\pm}(t). $$

Similarly, we can also show that $V_{\pm}(T)$ is the left derivative of $v_{\pm}(t)$ at $t=T$. Since $\delta>0$ is arbitrary, $v'_{\pm}(t)=V_{\pm}(t), \forall t\in(0,T]$. By Proposition \ref{Vol1}, it follows that $v_{\pm}\in C^1(0,T]$. Since $T>0$ is also arbitrary, $v_{\pm}\in C^1(0,\infty)$.
\end{proof}

Since $v_{\pm}\in C^1(0,\infty)$, we can check that $u\in C^{2,1}([0,1]\times (0,\infty))$.

Therefore, in view of Proposition \ref{C1} and the above regularity of $u$, the boundary conditions can be rewritten in a strong form as follows.
\begin{cor}
For any $t>0$, $v_{\pm}$ satisfy
\begin{equation}
\dfrac{d}{dt}v_-(t)=\dfrac{1}{2}u_r(0,t), \quad\quad \dfrac{d}{dt}v_+(t)=-\dfrac{1}{2}u_r(1,t), \quad\quad \dlim_{t\to 0}v_{\pm}(t) = v_{0,\pm}.
\end{equation}
\end{cor}
\subsection{Uniqueness of the solutions to the free boundary problem}
Notice that we have just identified the limits of the sequences $\rho^{\eps}([\eps^{-1}r], \eps^{-2}t)$, $\rho^{\eps}_{\pm}(\eps^{-2}t)$ up to a subsequence so far. 

Let us now consider other subsequences $\tilde{\rho}^{\eps_m}_{\pm}$ such that they converge uniformly to $\hat{v}_{\pm}$, respectively. It gives us the corresponding limit $\hat{u}$ of the subsequence $\rho^{\eps_m}$ given by the same expression as in \eqref{u}, where $v_{\pm}$ are replaced by $\hat{v}_{\pm}$. Again from the analysis in Chapter 6 of \cite{C}, we know that $\hat{u}$ also solves the boundary value problem \eqref{2.9} with the boundary values $\hat{v}_{\pm}$ and the initial datum $u_0$. On the other hand, using the same argument as in section \ref{bdrylimit} for the subsequence $\rho^{\eps_m}$, it follows that the boundary conditions as in \eqref{bdry} hold true for $\hat{v}_{\pm}$. 

We claim that for all $t\geq 0$, $v_-(t)=\hat{v}_-(t), v_+(t)=\hat{v}_+(t)$, so that for all $ t\geq 0$ and $r\in[0,1]$, $u(r,t)=\hat{u}(r,t)$. Indeed, let us denote $\bar{v}_{\pm}=v_{\pm}-\hat{v}_{\pm}$. We observe that $v_{\pm}$ inherit the boundedness and continuity from the sequence $\rho_{\pm}^{\eps}$. By \eqref{first} and \eqref{second}, it follows that
$$\bar{v}_-(t)=-\dint_{[0,t)}\Theta(0,t-s)\bar{v}_-(s)\,ds+\dint_{[0,t)}\Theta(1,t-s)\bar{v}_+(s)\,ds,$$
and similarly,
$$ \bar{v}_+(t)= \dint_{[0,t)}\Theta(1,t-s)\bar{v}_-(s)\,ds-\dint_{[0,t)}\Theta(0,t-s)\bar{v}_+(s)\,ds. $$


We denote $\mathbf{V}(t)=\bar{v}_-(t)+\bar{v}_+(t)$. The two above expressions give us
$$ \mathbf{V}(t)=\dint_{[0,t)} \big(\Theta(1,t-s)-\Theta(0,t-s)\big)\mathbf{V}(s)\,ds. $$

Applying Gronwall's inequality, we obtain that $\mathbf{V}(t)=0, \forall t\geq 0$. Since for any $t\geq 0$, $\bar{v}_{\pm}(t)\in[0,1]$, we arrive at our above claim.

Therefore, we have shown that the sequences $\rho^{\eps}([\eps^{-1}r], \eps^{-2}t), \rho^{\eps}_{\pm}(\eps^{-2}t)$ converge to $u(r,t)$ and $v_{\pm}(t)$, respectively, as $\eps$ goes to $0$. This completes the proof of Theorem \ref{PDE}.
\section{Proof of Theorem \ref{POC}}\label{PrPOC}
We denote $\Sigma_N:=\mathcal{S}_- \cup\Lambda_N\cup \mathcal{S}_+$. Let us consider the continuous time random walk $Y$ on $\Sigma_N$ with the generator given by
\begin{align}\label{RWAB}
&L^Yf(z)=\dfrac{1}{2N}[f(1)-f(z)], \text{ for } z\in\mathcal{S}_-\notag\\
&L^Yf(1)=\dfrac{1}{2}[f(2)-f(1)]+\dfrac{1}{2N}\dsum_{z\in\mathcal{S}_-}[f(z)-f(1)]\notag\\
&L^Yf(x)=\dfrac{1}{2}[f(x+1)+f(x-1)-2f(x)], \text{ for } 2\leq x\leq N-1\notag\\
&L^Yf(N)=\dfrac{1}{2}[f(N-1)-f(N)]+\dfrac{1}{2N}\dsum_{z\in\mathcal{S}_+}[f(z)-f(N)]\notag\\
&L^Yf(z)=\dfrac{1}{2N}[f(N)-f(z)], \text{ for } z\in\mathcal{S}_+.
\end{align}

Then we can obtain the so-called duality between our current exclusion process with density reservoirs at the boundaries and the continuous time random walk $Y$, namely
\begin{equation}\label{dualitynew}
\mathbf{E}^{*,\eps}\big[\eta(x,t)\big]=E_{x}\mathbf{E}^{*,\eps}\big[\eta(Y(t),0)\big], \forall x\in\Sigma_N.
\end{equation}

We call $\Gamma_N$ the set of non ordered couples of adjacent sites in $\Lambda_N$. Let us consider the process $(X_1, X_2)$ of two stirring walks with the generator $\mathbb{L}$ that acts on functions $f$ on $\{(x_1,x_2)\in \Sigma_N^2, x_1\neq x_2\}$ as follows
$$ \mathbb{L}f(x_1, x_2)=\dsum_{\{z,z'\}} c_{\{z,z'\}}\big[f\big((x_1,x_2)^{\{z,z'\}}\big)-f(x_1,x_2)\big],$$
where
\begin{itemize}[label={$\bullet$}]
\item The bond ${\{z,z'\}}\in \{\{z,1\}, z\in\mathcal{S}_-\} \cup\Gamma_N\cup \{\{N,z\}, z\in\mathcal{S}_+\}$;
\item $c_{\{z,1\}}=\dfrac{1}{2N}$, for $z\in\mathcal{S}_-$;$\quad c_{\{N, z\}}=\dfrac{1}{2N}$, for $z\in\mathcal{S}_+$;
\item $c_{\{z,z'\}}=\dfrac{1}{2}$, for $\{z,z'\}\in\Gamma_N$;
\item $(x_1,x_2)^{\{z,z'\}}=\big(x_1^{\{z,z'\}},x_2^{\{z,z'\}}\big)$ and for $i=\overline{1,2}$,
$$x_i^{\{z,z'\}}=
\begin{cases}
x_i, &\text{ if } z,z'\neq x_i,\\
z, &\text{ if } z'=x_i,\\
z', &\text{ if } z=x_i.
\end{cases}
$$
\end{itemize} 

Denote by $\mathbb{P}_{(x_1,x_2)}, \mathbb{E}_{(x_1,x_2)}$ the law and expectation of the process $(X_1, X_2)$ starting from $(x_1,x_2)$. Then for all $0\leq s<t$,
$$ \dfrac{d}{ds}\mathbb{E}_{(x_1,x_2)}\mathbf{E}^{*,\eps}\Big[\eta(X_1(t-s),s)\,\eta(X_2(t-s),s)\Big]=0. $$

It implies that
\begin{equation}\label{duality2}
\mathbf{E}^{*,\eps}\big[\eta(x_1,t)\,\eta(x_2,t)\big]=\mathbb{E}_{(x_1,x_2)}\mathbf{E}^{*,\eps}\Big[\eta(X_1(t),0)\,\eta(X_2(t),0)\Big].
\end{equation}

Let $Y_1,Y_2,Y$ be independent random walks having the generators given by \eqref{RWAB}. By using the integration by parts formula in Proposition 1.7 in Chapter VIII, \cite{L}, we can obtain that for any $t>0$,
\begin{align}\label{ibp}
&\mathbb{E}_{(x_1,x_2)}\big[f(Y_1(t), Y_2(t))\big]-\mathbb{E}_{(x_1,x_2)}\big[f(X_1(t), X_2(t))\big]\notag\\
&\quad=\dint_0^t\mathbb{E}_{(x_1,x_2)}\big[g(X_1(t-s), X_2(t-s))\big]\,ds, 
\end{align}
where
\begin{align*}
 &g(x_1,x_2)\\
=&\mathbf{1}_{\substack{|x_1-x_2|=1\\1\leq x_1,x_2\leq N}}\,\dfrac{1}{2}\big\{\mathbb{E}_{(x_1,x_1)}\big[f(Y_1(s), Y_2(s))\big]+\mathbb{E}_{(x_2,x_2)}\big[f(Y_1(s), Y_2(s))\big]\\
&\hskip7cm-2\mathbb{E}_{(x_1,x_2)}\big[f(Y_1(s), Y_2(s))\big]\big\}\\
+&\dsum_{z\in\mathcal{S}_-}\mathbf{1}_{x_1,x_2\in\{z,1\}}\,\dfrac{1}{2N}\big\{\mathbb{E}_{(1,1)}\big[f(Y_{1,s}, Y_2(s))\big]+\mathbb{E}_{(z,z)}\big[f(Y_1(s), Y_2(s))\big]\\
&\hskip7cm-2\mathbb{E}_{(1,z)}\big[f(Y_1(s), Y_2(s))\big]\big\}\\
+&\dsum_{z\in\mathcal{S}_+}\mathbf{1}_{x_1,x_2\in\{N,z\}}\dfrac{1}{2N}\big\{\mathbb{E}_{(N,N)}\big[f(Y_1(s), Y_2(s))\big]+\mathbb{E}_{(z,z)}\big[f(Y_1(s), Y_2(s))\big]\\
&\hskip7cm-2\mathbb{E}_{(N,z)}\big[f(Y_1(s), Y_2(s))\big]\big\}.
\end{align*}
For any subset $A$ of $\Sigma_N$, choosing $f(z_1,z_2)=\mathbf{1}_{(z_1,z_2)\in A\times A}$ in \eqref{ibp} gives us the Liggett's inequality
\begin{equation}\label{Lig}
\mathbb{P}_{(x_1,x_2)}\big((X_1(t), X_2(t))\in A\times A\big)\leq \mathbb{P}_{(x_1,x_2)}\big((Y_1(t), Y_2(t))\in A\times A\big).
\end{equation}
Moreover, by choosing $f(z_1,z_2)=\mathbf{E}^{*,\eps}\big[\eta(z_1,0)\eta(z_2,0)\big]$ in \eqref{ibp}, we obtain that

$$\mathbb{E}_{(x_1,x_2)}\mathbf{E}^{*,\eps}\big[\eta(Y_1(t),0)\eta(Y_2(t),0)\big]-\mathbb{E}_{(x_1,x_2)}\mathbf{E}^{*,\eps}\big[\eta(X_1(t),0)\eta(X_2(t),0)\big]$$
$$=\dint_0^t\mathbb{E}_{(x_1,x_2)}\big[\bar{g}(X_1(t-s),X_2(t-s))\big]\,ds,$$
where
\begin{align*}
&\bar{g}(x_1,x_2)\\
=&\mathbf{1}_{\substack{|x_1-x_2|=1\\1\leq x_1,x_2\leq N}}\dfrac{1}{2}\big\{E_{x_1}\mathbf{E}^{*,\eps}\big[\eta(Y(s),0)\big]-E_{x_2}\mathbf{E}^{*,\eps}\big[\eta(Y(s),0)\big]\big\}^2\\
&+\dsum_{z\in\mathcal{S}_-}\mathbf{1}_{x_1, x_2\in\{z,1\}}\dfrac{1}{2N}\big\{E_z\mathbf{E}^{*,\eps}\big[\eta(Y(s),0)\big]-E_1\mathbf{E}^{*,\eps}\big[\eta(Y(s),0)\big]\big\}^2\\
&+\dsum_{z\in\mathcal{S}_+}\mathbf{1}_{x_1, x_2\in\{N,z\}}\dfrac{1}{2N}\big\{E_N\mathbf{E}^{*,\eps}\big[\eta(Y(s),0)\big]-E_z\mathbf{E}^{*,\eps}\big[\eta(Y(s),0)\big]\big\}^2.
\end{align*}
Then the conclusion of Theorem \ref{POC} will follow from the following two steps.

In the first step, we will look for the upper bounds of the terms of the form $\big|E_{y_1}\mathbf{E}^{*,\eps}\big[\eta(Y(s),0)\big]-E_{y_2}\mathbf{E}^{*,\eps}\big[\eta(Y(s),0)\big]\big|$, for any bond $\{y_1, y_2\}\in \{\{z,1\}, z\in\mathcal{S}_-\} \cup\Gamma_N\cup \{\{N,z\}, z\in\mathcal{S}_+\}$. 
\begin{lem}\label{boundsquare}
For any $s>0$,
\begin{equation}\label{Sleft}
\big|E_z\mathbf{E}^{*,\eps}\big[\eta(Y(s),0)\big]-E_1\mathbf{E}^{*,\eps}\big[\eta(Y(s),0)\big]\big|\leq \dfrac{1}{N}+ \dfrac{C}{\sqrt{s}},\quad \forall z\in\mathcal{S}_- 
\end{equation}
\begin{equation}\label{Sright}
\big|E_N\mathbf{E}^{*,\eps}\big[\eta(Y(s),0)\big]-E_z\mathbf{E}^{*,\eps}\big[\eta(Y(s),0)\big]\big|\leq \dfrac{1}{N}+\dfrac{C}{\sqrt{s}},\quad \forall z\in\mathcal{S}_+.
\end{equation}
For any $\alpha\in (0,\dfrac{1}{6}), s>0$,
\begin{equation}\label{channel}
\big|E_y\mathbf{E}^{*,\eps}\big[\eta(Y(s),0)\big]-E_{y+1}\mathbf{E}^{*,\eps}\big[\eta(Y(s),0)\big]\big|\leq \dfrac{C}{N^{1-\alpha}}, \quad \forall y,y+1\in\Lambda_N.
\end{equation}
\end{lem}
\begin{proof} 
To verify \eqref{Sleft} and \eqref{Sright}, we will use the same argument. So only the estimate \eqref{Sright} is proved here. Let $Z'_1, Z'_2$ be two random walks whose generators are given by \eqref{RWAB}. We construct the coupling $\underline{Z'}=(Z'_1, Z'_2)$ as follows: $Z'_1$ and $Z'_2$ start from $N$ and $z\in\mathcal{S}_+$, move independently up to the first time when they meet each other or both of them are in the same reservoir, and from that time they move in the same way. Let us denote $\mathcal{C'}(s)$ the event that $Z'_1$ and $Z'_2$ do not meet each other and both of them are not in the same reservoir up to time $s$. Therefore, for any $s>0$ and $z\in\mathcal{S}_+$,
\begin{align*}
&\big|E_N\mathbf{E}^{*,\eps}\big[\eta(Y(s),0)\big]-E_z\mathbf{E}^{*,\eps}\big[\eta(Y(s),0)\big]\big|\\
=&\Big|\mathbb{E}_{(N,z)}\Big[\big[\mathbf{E}^{*,\eps}\big[\eta(Z'_1(s),0)\big]-\mathbf{E}^{*,\eps}\big[\eta(Z'_2(s),0)\big]\big]\mathbf{1}_{\mathcal{C'}(s)}\Big]\Big|\\
\leq &\mathbb{P}_{(N,z)}(\mathcal{C'}(s))\\
\leq &P_N(\tau^{Y^{\rm sp}}_{N+1}>s)+P_N(\tau_0^{Y^{\rm sp}}<\tau_{N+1}^{Y^{\rm sp}})\\
\leq &\dfrac{C}{\sqrt{s}}+\dfrac{1}{N},
\end{align*}
where $Y^{\rm sp}$ is a simple symmetric random walk on $\mathbb{Z}$.

Next we will show the estimate \eqref{channel}. For two random walks $Z_1,Z_2$ having generators given by \eqref{RWAB}, we construct the coupling $\underline{Z}=(Z_1, Z_2)$ as follows: $Z_1$ and $Z_2$ start from $y$ and $y+1$, move such that they are closed of distance $1$ up to the first time when one of them reach the reservoir, from that time they move independently until when they first meet each other or both of them are in the same reservoir, and then they move in the same way. Call $\mathcal{C}(s)$ the event that $Z_1$ and $Z_2$ do not meet each other and both of them are not in the same reservoir up to time $s$. For $y, y+1\in\Lambda_N$ and $s>0$,
\begin{align*}
&\big|E_y\mathbf{E}^{*,\eps}\big[\eta(Y(s),0)\big]-E_{y+1}\mathbf{E}^{*,\eps}\big[\eta(Y(s),0)\big]\big|\\
=&\Big|\mathbb{E}_{(y,y+1)}\Big[\big[\mathbf{E}^{*,\eps}\big[\eta(Z_1(s),0)\big]-\mathbf{E}^{*,\eps}\big[\eta(Z_2(s),0)\big]\big]\mathbf{1}_{\mathcal{C}(s)}\Big]\Big|.
\end{align*}

For $\alpha \in (0,\dfrac{1}{6})$, let us consider the following cases.

Case 1: $\mathcal{A}=\{\underline{Z}$ touches $\mathcal{S}_-$ or $\mathcal{S}_+$ for the first time prior to $s-N^{2-\alpha}\}$.

Set $\tau_{\mathcal{S}_-}^{Z_1}=\inf\{\sigma>0, Z_1(\sigma)\in\mathcal{S}_-\}$ and $\tau_{\mathcal{S}_+}^{Z_2}=\inf\{\sigma>0, Z_2(\sigma)\in\mathcal{S}_+\}$. Then
\begin{align*}
&\Big|\mathbb{E}_{(y,y+1)}\Big[\big[\mathbf{E}^{*,\eps}\big[\eta(Z_1(s),0)\big]-\mathbf{E}^{*,\eps}\big[\eta(Z_2(s),0)\big]\big]\mathbf{1}_{\mathcal{C}(s)}\mathbf{1}_{\mathcal{A}}\Big]\Big|\\
\leq \,\,\,&\mathbb{P}_{(y,y+1)}(\mathcal{C}(s)\cap\mathcal{A})\\
=\,\,\,&\dint_0^{s-N^{2-\alpha}}\mathbb{P}_{(y,y+1)}(\tau_{\mathcal{S}_-}^{Z_1}\wedge\tau_{\mathcal{S}_+}^{Z_2}=\sigma)\,\mathbb{P}_{(y,y+1)}(\mathcal{C}(s)|\tau_{\mathcal{S}_-}^{Z_1}\wedge\tau_{\mathcal{S}_+}^{Z_2}=\sigma)\,d\sigma\\
=\,\,\,&\dint_0^{s-N^{2-\alpha}}\mathbb{P}_{(y,y+1)}(\tau_{\mathcal{S}_-}^{Z_1}\wedge\tau_{\mathcal{S}_+}^{Z_2}=\sigma)\Big(\dfrac{1}{N}+\dfrac{C}{\sqrt{s-\sigma}}\Big)\,d\sigma\leq \dfrac{C}{N^{1-\frac{\alpha}{2}}}.
\end{align*}

Case 2: $\mathcal{B}=\{\underline{Z}$ does not touch $\mathcal{S}_{\pm}$ for the first time prior to $s-N^{2-\alpha}\}$.

Set
$$\mathcal{D}=\{\underline{Z} \text{ does not touch $\mathcal{S}_{\pm}$ after } s-N^{2-\alpha}\}.$$

By the assumption that $u_0\in C^1(0,1)$, we get
\begin{align*}
&\Big|\mathbb{E}_{(y,y+1)}\Big[\big[\mathbf{E}^{*,\eps}\big[\eta(Z_1(s),0)\big]-\mathbf{E}^{*,\eps}\big[\eta(Z_2(s),0)\big]\big]\mathbf{1}_{\mathcal{C}(s)}\mathbf{1}_{\mathcal{B}\cap \mathcal{D}}\Big]\Big|\\
=&\Big|\mathbb{E}_{(y,y+1)}\big[\big[u_0(\eps Z_1(s))-u_0(\eps Z_2(s))\big]\mathbf{1}_{\mathcal{C}(s)}\mathbf{1}_{\mathcal{B}\cap \mathcal{D}}\big]\Big|\\
\leq &\mathbb{E}_{(y,y+1)}\big[C\big|\eps Z_1(s)-\eps Z_2(s)\big|\big]\leq \dfrac{C}{N}.
\end{align*}

On the other hand, we have
\begin{align*}
&\Big|\mathbb{E}_{(y,y+1)}\Big[\big[\mathbf{E}^{*,\eps}\big[\eta(Z_1(s),0)\big]-\mathbf{E}^{*,\eps}\big[\eta(Z_2(s),0)\big]\big]\mathbf{1}_{\mathcal{C}(s)}\mathbf{1}_{\mathcal{B}\cap \mathcal{D}^c}\Big]\Big|\\
\leq &\mathbb{P}_{(y,y+1)}(\mathcal{C}(s)\cap\mathcal{B}\cap \mathcal{D}^c)\\
\leq & \dsum_{z,z+1\in\Lambda_N}\mathbb{P}_{(y,y+1)}(\underline{Z}(s-N^{2-\alpha})=(z,z+1))\,\mathbb{P}_{(z,z+1)}(\mathcal{C}(N^{2-\alpha})\cap \mathcal{E}),
\end{align*}
where $\mathcal{E}=\{\underline{Z} \text{ touches $\mathcal{S}_-$ or $\mathcal{S}_+$ in $[0,N^{2-\alpha}]$}\}$. For all $z,z+1\in \Lambda_N$, since
\begin{align*}
\mathds{A}(\sigma):=&\mathbb{P}_{(z,z+1)}(\tau_{\mathcal{S}_-}^{Z_1}\wedge\tau_{\mathcal{S}_+}^{Z_2}=\sigma)\\
\leq &\mathbb{P}_{(z,z+1)}(\tau_{\mathcal{S}_-}^{Z_1}=\sigma,\tau_{\mathcal{S}_-}^{Z_1}<\tau_{\mathcal{S}_+}^{Z_1})+\mathbb{P}_{(z,z+1)}(\tau_{\mathcal{S}_+}^{Z_2}=\sigma,\tau_{\mathcal{S}_+}^{Z_2}<\tau_{\mathcal{S}_-}^{Z_2})\\
= &P_z(\tau_{\mathcal{S}_-}^{Z_1}=\sigma,\tau_{\mathcal{S}_-}^{Z_1}<\tau_{\mathcal{S}_+}^{Z_1})+P_{z+1}(\tau_{\mathcal{S}_+}^{Z_2}=\sigma,\tau_{\mathcal{S}_+}^{Z_2}<\tau_{\mathcal{S}_-}^{Z_2})\\
\leq &P_z(\tau_0^{Y^{\rm sp}}=\sigma)+P_{z+1}(\tau_{N+1}^{Y^{\rm sp}}=\sigma)\\
\leq &\dfrac{Cz}{\sigma^{3/2}}+\dfrac{C(N+1-(z+1))}{\sigma^{3/2}}\\
\leq &\dfrac{CN}{\sigma^{3/2}},
\end{align*}
and
\begin{align*}
&\mathds{B}(\sigma):=\mathbb{P}_{(z,z+1)}(\mathcal{C}(N^{2-\alpha})\cap \mathcal{E}|\tau_{\mathcal{S}_-}^{Z_1}\wedge\tau_{\mathcal{S}_+}^{Z_2}=\sigma)\\
=&\mathbf{1}_{\tau_{\mathcal{S}_-}^{Z_1}<\tau_{\mathcal{S}_+}^{Z_2}}\dsum_{u\in\mathcal{S}_-}\mathbf{1}_{\tau_{\mathcal{S}_-}^{Z_1}=\tau_u^{Z_1}}\,\mathbb{P}_{(u,1)}(\mathcal{C}(N^{2-\alpha}-\sigma))\\
&\quad+\mathbf{1}_{\tau_{\mathcal{S}_+}^{Z_2}<\tau_{\mathcal{S}_-}^{Z_1}}\dsum_{u\in\mathcal{S}_+}\mathbf{1}_{\tau_{\mathcal{S}_+}^{Z_2}=\tau_u^{Z_2}}\,\mathbb{P}_{(N,u)}(\mathcal{C}(N^{2-\alpha}-\sigma))\\
\leq &\dfrac{1}{N}+P_1(\tau_0^{Y^{\rm sp}}>N^{2-\alpha}-\sigma)+P_N(\tau_{N+1}^{Y^{\rm sp}}>N^{2-\alpha}-\sigma)\\
\leq &\dfrac{1}{N}+\dfrac{C}{\sqrt{N^{2-\alpha}-\sigma}},
\end{align*}
then
\begin{align*}
&\mathbb{P}_{(z,z+1)}(\mathcal{C}(N^{2-\alpha})\cap \mathcal{E})=\dint_0^{N^{2-\alpha}}\mathds{A}(\sigma)\,\mathds{B}(\sigma)\,d\sigma\\
\leq &\Big(\dfrac{1}{N}+\dfrac{C}{\sqrt{N^{2-\alpha}}}\Big)\dint_0^{\frac{N^{2-\alpha}}{2}}\mathds{A}(\sigma)\,d\sigma+\dint_{\frac{N^{2-\alpha}}{2}}^{N^{2-\alpha}}\dfrac{CN}{\sigma^{3/2}}\Big(\dfrac{1}{N}+\dfrac{C}{\sqrt{N^{2-\alpha}-\sigma}}\Big)\,d\sigma\\
\leq &\dfrac{C}{N^{1-\alpha}}.
\end{align*}

It follows that
$$ \Big|\mathbb{E}_{(y,y+1)}\Big[\big[\mathbf{E}^{*,\eps}\big[\eta(Z_1(s),0)\big]-\mathbf{E}^{*,\eps}\big[\eta(Z_2(s),0)\big]\big]\mathbf{1}_{\mathcal{C}(s)}\mathbf{1}_{\mathcal{B}\cap \mathcal{D}^c}\Big]\Big|\leq \dfrac{C}{N^{1-\alpha}}. $$

Hence, the above cases imply the estimate \eqref{channel}.
\end{proof}
In the second step, we will verify the following bound.
\begin{lem}\label{boundprob}
For $\theta:=N^2(t-s)$,
$$\mathbb{P}_{(x_1,x_2)}(|X_1(\theta)-X_2(\theta)|=1, 1\leq X_1(\theta), X_2(\theta)\leq N)\leq\dfrac{C}{\sqrt{\theta}}+\dfrac{C}{N^{1/3}}.$$
\end{lem}
\begin{proof}
By Liggett's inequality \eqref{Lig}, we obtain that
\begin{align*}
&\mathbb{P}_{(x_1,x_2)}(|X_1(\theta)-X_2(\theta)|=1, 1\leq X_1(\theta), X_2(\theta)\leq N)\\
\leq &\dsum_{1\leq y,y+1\leq N}P_{x_1}(Y_1(\theta)\in\{y,y+1\})P_{x_2}(Y_2(\theta)\in\{y,y+1\}).
\end{align*}

To obtain the desired estimate, it suffices to bound $P_x(Y(\theta)=y)$ for $y\in\Lambda_N$. For any $\delta>0$,
\begin{align*}
&P_x(Y(\theta)=y)=P_x(Y(\theta)=y, \tau^Y_{\mathcal{S}_-}\wedge\tau^Y_{\mathcal{S}_+}>\theta)+P_x(Y(\theta)=y, \tau^Y_{\mathcal{S}_-}\wedge\tau^Y_{\mathcal{S}_+}\leq\theta)\\
\leq &P_x(Y^{\rm sp}(\theta)=y)\\
&+\dint_0^{\theta-N^{1-\delta}}P_x(Y(\theta)=y|\sigma=\text{the last time $Y$ touches }{\mathcal{S}_{\pm}})\,dP_x(\sigma)\\
&+\dint^{\theta}_{\theta-N^{1-\delta}}P_x(Y(\theta)=y|\sigma=\text{the first time $Y$ touches } {\mathcal{S}_{\pm}} \text{ after }\theta-N^{1-\delta})\,dP_x(\sigma)\\
&=: \mathbf{A}+\mathbf{B}+\mathbf{C}.
\end{align*}

By choosing $\delta=1/3$, the claim of our lemma follows from the upper bounds of $\mathbf{A},\mathbf{B}$ and $\mathbf{C}$. Namely, by the Local Central Limit Theorem, we have $\mathbf{A}\leq \dfrac{C}{\sqrt{\theta}}$. Moreover, since $Y$ spends at $\mathcal{S}_{\pm}$ in exponential time of parameter $1/(2N)$ then $\mathbf{C}\leq \dfrac{1}{2N^{\delta}}$. On the other hand, $\mathbf{B}$ is equal to
$$\dint_0^{\theta-N^{1-\delta}}P_x(Y(\theta)=y|Y(\sigma^+)\in\{1,N\}; 1\leq Y(v)\leq N, \forall v\in(\sigma,\theta])\,dP_x(\sigma)$$
$$\leq\dint_0^{\theta-N^{1-\delta}}\sum_{z\in\{1,N\}}P_z(Y^{\rm sp}(\theta-\sigma)=y| 1\leq Y^{\rm sp}(v)\leq N, \forall v\in[0,\theta-\sigma])\,dP_x(\sigma)$$
$$=\dint^{\theta}_{N^{1-\delta}}\dsum_{z\in\{1,N\}}\dfrac{P_z(Y^{\rm sp}(\sigma)=y;1\leq Y^{\rm sp}(v)\leq N, \forall v\in[0,\sigma])}{P_z(1\leq Y^{\rm sp}(v)\leq N, \forall v\in[0,\sigma])}\,dP_x(\theta-\sigma).$$

We now would like to look for the upper bound of the fraction in the above integrand for $z=N$. The fraction with $z=1$ can be bounded in a similar way.

To bound the numerator of the fraction from above, we have
\begin{align*}
&P_N(Y^{\rm sp}(\sigma)=y;1\leq Y^{\rm sp}(v)\leq N, \forall v\in[0,\sigma])\\
\leq\,\,& P_N(\tau^{Y^{\rm sp}}_0\wedge\tau^{Y^{\rm sp}}_{N+1}>\sigma)\\
\leq\,\, &P_N(\tau^{Y^{\rm sp}}_{N+1}>\sigma/2)\cdot P_1(\tau^{Y^{\rm sp}}_0>\sigma/2)\leq \dfrac{C}{\sigma}.
\end{align*}

By using the technique presented in \cite{DPTV3}, for $N^{2-\delta}\leq \sigma\leq\theta$, we can show that the denominator of the fraction is bounded from below by 
\begin{align*}
&P_N\big(\{1\leq Y^{\rm sp}(N^2\kappa)\leq N/100; 1\leq Y^{\rm sp}(v)\leq N, \forall v\in[0,N^2\kappa]\}\\
&\,\,\,\,\,\,\,\,\,\,\,\,\,\bigcap_{i=2}^m\{1\leq Y^{\rm sp}(iN^2\kappa)\leq N/100; 1\leq Y^{\rm sp}(v)\leq N, \forall v\in[i-1,i]N^2\kappa\}\big)\\
\geq&\,\,\, \dfrac{C}{N}\text{ (for $\kappa$ small enough, $\sigma=mN^2\kappa,m\in\mathbb{Z}$)}.
\end{align*}

For $N^{1-\delta}\leq \sigma\leq N^{2-\delta}$, the denominator of the fraction can be rewritten by
$$P_N(\tau^{Y^{\rm sp}}_0\wedge\tau^{Y^{\rm sp}}_{N+1}>\sigma)\geq P_N(\tau^{Y^{\rm sp}}_{N+1}>\sigma)-P_N(\tau^{Y^{\rm sp}}_{N+1}>\tau^{Y^{\rm sp}}_0)\geq \dfrac{C}{\sqrt{\sigma}}.$$

It implies that $\mathbf{B}\leq\dfrac{C}{\sqrt{N^{1-\delta}}}$. 
\end{proof}

Lemma \ref{boundsquare} and Lemma \ref{boundprob} lead us to the fact that
\begin{align*}
\dlim_{N\to\infty}\bigg|&\mathbb{E}_{(x_1,x_2)}\mathbf{E}^{*,\eps}\big[\eta(X_1(N^2t),0)\eta(X_2(N^2t),0)\big]\\
&-E_{x_1}\mathbf{E}^{*,\eps}[\eta(Y_1(N^2t),0)]\,E_{x_2}\mathbf{E}^{*,\eps}[\eta(Y_2(N^2t),0)]\bigg|=0.
\end{align*}

Recall the equalities \eqref{dualitynew} and \eqref{duality2}, we arrive at the conclusion of our theorem.

{\bf Acknowledgements.} I greatly appreciate Prof. Errico Presutti for suggesting the problem and offering me a large number of useful ideas. I also would like to express my gratitude to Maria Eulalia Vares, Lorenzo Bertini, Paolo Butta, Pablo Ferrari and Frank Redig for their valuable comments and suggestions.

\bibliographystyle{amsalpha}

\end{document}